\documentclass[11pt]{article}
\usepackage{amsmath, amssymb, amsthm, verbatim,enumerate,bbm}
\usepackage{tikz}
\usetikzlibrary{calc}
\usepackage[margin=1in]{geometry}
\usepackage[norelsize]{algorithm2e}

\allowdisplaybreaks

\theoremstyle{plain}
\newtheorem{theorem}{Theorem}[section]
\newtheorem{lemma}[theorem]{Lemma}
\newtheorem{claim}[theorem]{Claim}

\newtheorem{remark}[theorem]{Remark}
\newtheorem{definition}[theorem]{Definition}

\newcommand{\Bin}{\ensuremath{\textrm{Bin}}}

\newcommand{\dnp}{\ensuremath{D(n,p)}}
\newcommand{\gnp}{\ensuremath{G(n,p)}}
\newcommand{\ham}{\ensuremath{\mathcal H \mathcal A \mathcal M}}
\newcommand{\PMinDegall}{(P\ref{property:min-degree-all})}
\newcommand{\PMinDeg}{(Q\ref{property:min-degree})}
\newcommand{\PXEdges}{(P\ref{property:few-edges-X})}
\newcommand{\PXYEdges}{(P\ref{property:edges-X-Y})}
\newcommand{\PSize}{(Q\ref{property:size})}

\begin{document}

\thispagestyle{empty} 
\begin{center}
\LARGE Robust hamiltonicity of random directed graphs
\vspace{8mm}

\Large{

\begin{tabular}{ccc}
Asaf Ferber &   Rajko Nenadov &  Andreas Noever \\
{\small{asaf.ferber@inf.ethz.ch}} &
{\small{rnenadov@inf.ethz.ch}} &
{\small{anoever@inf.ethz.ch}}
\end{tabular} \vspace{3mm}

\begin{tabular}{cc}
   Ueli Peter &   Nemanja \v Skori\'c   \\
   {\small{upeter@inf.ethz.ch}} &
  {\small{nskoric@inf.ethz.ch}}
\end{tabular}
}
\vspace{5mm}

\large
  Institute of Theoretical Computer Science \\
  ETH Zurich, 8092 Zurich, Switzerland
\vspace{8mm}

\end{center}

\begin{abstract}
In his seminal paper from 1952 Dirac showed that the complete graph
on $n\geq 3$ vertices remains Hamiltonian even if we allow an
adversary to remove $\lfloor n/2\rfloor$ edges touching each vertex.
In 1960 Ghouila-Houri obtained an analogue statement for digraphs by
showing that every directed graph on $n\geq 3$ vertices with minimum
in- and out-degree at least $n/2$ contains a directed Hamilton
cycle. Both statements quantify the robustness of complete graphs
(digraphs) with respect to the property of containing a Hamilton
cycle.
\\
\indent
A natural way to generalize such results to arbitrary graphs
(digraphs) is using the notion of \emph{local resilience}. The local
resilience of a graph (digraph) $G$ with respect to a property
$\mathcal{P}$ is the maximum number $r$ such that $G$ has the
property $\mathcal{P}$ even if we allow an adversary to remove an
$r$-fraction of (in- and out-going) edges touching each vertex. The
theorems of Dirac and Ghouila-Houri state that the local resilience
of the complete graph and digraph with respect to Hamiltonicity is $1/2$. Recently, this
statements have been generalized to random settings.  Lee and
Sudakov (2012) proved that the local resilience of a random graph
with edge probability $p=\omega\left(\log n /n\right)$ with respect to Hamiltonicity
 is $1/2\pm o(1)$. For random
directed graphs, Hefetz, Steger and Sudakov (2014+) proved an
analogue statement, but only for edge probability
$p=\omega\left(\log n/\sqrt{n}\right)$. In this paper we
significantly improve their result to $p=\omega\left(\log^8 n/
n\right)$, which is optimal up to the polylogarithmic factor.

\end{abstract}
\newpage

\setcounter{page}{1}

\section{Introduction}
A \emph{Hamilton cycle} in a graph or a directed graph is a cycle
that passes through all the vertices of the graph exactly once, and
a graph is \emph{Hamiltonian} if it contains a Hamilton cycle.
Hamiltonicity is one of the central notions in graph theory,
and has been intensively studied by numerous researchers. It is well
known that the problem of whether a given graph contains a Hamilton
cycle is $\mathcal{NP}$-complete. In fact, Hamiltonicity was one of
Karp's 21 $\mathcal{NP}$-complete problems
\cite{karp1972reducibility}.

Since one can not hope for a general classification of Hamiltonian
graphs, as a consequence of Karp's result, there is a large interest in deriving properties that are
sufficient for Hamiltonicity. A classic result by Dirac from 1952
\cite{dirac1952some} states that every graph on $n\geq 3$ vertices
with minimum degree at least $n/2$ is Hamiltonian. This result is
tight as the complete bipartite graph with parts of sizes that
differ by one, $K_{m,m+1}$, is not Hamiltonian. Note that this
theorem answers the following question:  Starting with the complete
graph on $n$ vertices $K_n$, what is the maximal integer $\Delta$
such that for any subgraph $H$ of $K_n$ with maximum degree
$\Delta$, the graph $K_n-H$ obtained by deleting the edges of $H$
from $K_n$ is Hamiltonian? This question not only asks for a
sufficient condition for a graph to be Hamiltonian, it also asks for
a quantification for the ``local robustness" of the complete graph
with respect to Hamiltonicity.

A natural generalization of this question is to replace the complete graph with some arbitrary base graph.
Recently, questions of this type have drawn a lot of
attention under the notion of \emph{resilience}.

Roughly speaking, given a monotone increasing graph property $\mathcal P$ and a graph
or a digraph $G$ which satisfies $\mathcal P$, the {\em
resilience of $G$ with respect to $\mathcal P$} measures how much one must change $G$,
 in order to destroy $\mathcal P$.
Since one can destroy many natural properties by small changes (for
example, by isolating a vertex), it is natural to limit
the number of edges touching any vertex that one is allowed to
delete. This leads to the following definition of {\em local
resilience}.

\begin{definition}[Local resilience] \label{def:local}
Let $\mathcal P$ be a monotone increasing graph property. For a graph $G$, the local resilience is
\begin{align*}r(G, \mathcal{P}):=\min\{r:\; & \exists H\subseteq G \text{ such that } \\
& \forall_{v\in V(G)}  \,  d_H(v)\le r\cdot d_G(v)
 \text { and }\\  & G - H \text{ does not have $\mathcal P$}\},
\end{align*}
while for a digraph $G$ it is defined as
\begin{align*}
r(G, \mathcal{P}):= & \min\{r:\;  \exists H\subseteq G \text{ such that } 
 \forall v\in V(G) \\ 
 & d^+_H(v)\le r\cdot d^+_G(v)\text{ and } 
 d^-_H(v)\le r\cdot d^-_G(v) \\
 & \text {and } 
  G - H \text{ does not have $\mathcal P$}\}.
\end{align*}
\end{definition}
Sudakov and Vu initiated the systematic study of resilience of
random and pseudorandom graphs in \cite{SudakovVu}, and since then this
field has attracted substantial research interest (see e.g.\
\cite{BaloghEtAl,ben2011local,ben2011resilience,bottcher2009almost,FriezeKrivelevich,krivelevich2010resilient,lee2012dirac}).

Let us denote with $\ham$ the graph property of containing a Hamiltonian cycle (directed, in case of digraphs). Lee and Sudakov \cite{lee2012dirac} proved
that for $p=\omega\left(\log n/n\right)$, a
typical $G\sim \gnp$ satisfies $r(G, \ham) \in (1/2 \pm o(1))$. Note that this result is
asymptotically optimal not only with respect to the constant $1/2$
but also with respect to the probability $p$, since it is well known
that a typical graph $G\sim \gnp$ is not Hamiltonian for $p=o(\log
n/n)$ (see \cite{bollobas1998random}).

For a positive integer $n$ and $0\leq p=p(n)\leq 1$, let $\dnp$
denote the binomial probability space of random digraphs on the set
of vertices $[n]=\{1,\ldots,n\}$. That is, an element $D \sim \dnp$
is generated by including each of the $n(n-1)$ possible ordered
pairs of $[n]$ with probability $p$, independently at random. For this model, Frieze
\cite{frieze1988algorithm} showed that a typical digraph $D\sim
\dnp$ is Hamiltonian for $p\geq (\log n+ \omega(1))/n$. Therefore, it is natural to ask for an analogue to the result of Lee and
Sudakov \cite{lee2012dirac} for random digraphs with these
densities.

As a first step towards this goal, Hefetz, Steger and Sudakov proved
in \cite{hefetz2014random} the following theorem, which is
asymptotically optimal with respect to the resilience but far from optimal with respect to the edge probability.

\begin{theorem}[\cite{hefetz2014random}] \label{HSS}
Let $\beta>0$, let $n$ be a sufficiently large integer and let
$p=\omega\left(\log n/\sqrt{n}\right)$. Then w.h.p.\ a digraph $G\sim
\dnp$ satisfies $r(G, \ham) \in (1/2 \pm \beta)$.
\end{theorem}

In the proof of Theorem \ref{HSS} Hefetz, Steger and Sudakov
extensively used the Regularity Lemma and the fact that for
$p=\omega\left(\log n/\sqrt{n}\right)$, a typical digraph $G\sim
\dnp$ contains ``many" transitive triangles touching each vertex.
Therefore, generalizing it to smaller values of $p$ would at least require to replace triangles with some sparser gadgets.

In general, problems related to Hamilton cycles in digraphs are
known to be much harder than their counterparts in the undirected
setting, mainly since
the Pos\'a rotation-extension technique (see \cite{posa1976hamiltonian}) is,  
in its simplest form, not applicable to directed graphs.

In this paper we use the  absorbing method, initiated by R{\"o}dl, Ruci{\'n}ski and Szemer{\'e}di \cite{RodlRucinskiSzemeredi2008Absorbing},   combined with a
very nice and recent embedding argument of Montgomery \cite{montgomery2014embedding} to prove the
following theorem, which is optimal up to polylogarithmic factors.

\begin{theorem} \label{thm:main}
Let $\beta>0$, let $n$ be a sufficiently large integer and let
$p=\omega\left(\log ^8 n/ n\right)$. Then w.h.p.\ a digraph $G\sim
\dnp$ satisfies $r(G, \ham) \in (1/2 \pm \beta)$.
\end{theorem}
We want to remark that our proof can easily be turned into a simple and efficient randomized algorithm which finds a Hamilton cycle in a digraph with certain pseudorandom properties.

The paper is organized as follows. In Section \ref{sec:tools} we present auxiliary lemmas which are used throughout the paper. In Section \ref{sec:main_result} we give the definition of \emph{$(n, \alpha, p)$-pseudorandom} digraphs and state our main result (Theorem \ref{thm:main_pseudo}) concerning the Hamiltonicity of such digraphs. We then show how it implies Theorem \ref{thm:main} and furthermore derive the proof of Theorem \ref{thm:main_pseudo} using Connecting and Absorbing lemmas. 
In Section \ref{sec:connecting_lemma} we then give a proof of Connecting Lemma, and finally in Section \ref{sec:absorbing_lemma} we furthermore use it to prove Absorbing Lemma.

\subsection{Notation and definitions}

For an integer $n$, let $[n] =\{1, \ldots, n\}$ and for $a, b, c \in
\mathbb R$, let $(a\pm b)c=((a-b)c,(a+b)c)$.

Our graph theoretic notation is standard and follows that of
\cite{West}. In particular we use the following: Given a digraph $D$
we denote by $V(D)$ and $E(D)$ the sets of vertices and arcs of $D$,
respectively, and denote $v(D) = |V(D)|$ and $e(D) = |E(D)|$.
For a subset $S \subseteq V(D)$, we denote with $D[S]$ the subgraph
of $D$ induced by $S$. For two (not necessarily disjoint) subsets $X, Y \subseteq V(D)$, set $E_D(X,Y) := \{ (x, y)
\in E(D) : x \in X \text{ and } y \in Y \}$ and let $e_D(X, Y) =
|E_D(X,Y)|$. Furthermore, let $N^+_D(X,Y) = \{ y \in Y : x \in
X \text { and } (x,y) \in E(D) \}$ denote the set of all
out-neighbors of $X$ in $Y$ and let $N^-_D(X,Y) = \{ y \in Y : x \in
X \text { and } (y,x) \in E(D) \}$ denote the set of all
in-neighbors of $X$ in $Y$. 
Given a vertex $x \in V(D)$ and $\tau \in \{+,-\}$, we abbreviate $N^\tau_D(\{x\}, Y)$ to $N^\tau_D(x, Y)$ and
 define $d^{\tau}_D(x,Y)=|N_D^{\tau}(x,Y)|$ and $d_D^{\pm}(x, Y) = \min\{ d_D^+(x, Y), d_D^-(x, Y) \}$.
We omit the subscript $D$ whenever there is no risk of confusion.

For $\tau \in \{+, -\}$ we denote with $\bar \tau$ the opposite sign. Furthermore, for  $\sigma \in \{+,
-\}^{\ell}$ and $i \in [\ell]$, let $\sigma(i)$ denote $i$-th member
of the $\ell$-tuple $\sigma$, let $\sigma^i = (\sigma(1), \ldots, \sigma(i))$  
and let $\bar \sigma$ denote $(\bar \sigma(\ell), \ldots,  \bar \sigma(1))$.  
We call a sequence of vertices 
$P = v_1,\ldots,v_{\ell+1}$ a $\sigma$-walk if all the vertices are different, except that  $v_0$ and $v_{\ell + 1}$ can be the same vertex,
and if $v_{i+1} \in N^{\sigma(i)}(v_i)$ for all $1 \leq i \le \ell$. Moreover, we say that $P$ connects $v_1$ to $v_{\ell + 1}$ and call $v_1$ and $v_{\ell +1}$ its left and right endpoint, respectively. The $\sigma$-walk $P$ is additionally called an \emph{$v_1v_{\ell+1}$-path} if $v_1 \neq v_{\ell +1 }$ and  $\sigma(i) = +$, for all $1 \le i \le \ell$.


\section{Tools and preliminaries}
\label{sec:tools}

In this section we introduce tools used in the proofs of our results.


\subsection{Probabilistic tools}

We need to employ standard bounds on large deviations of random
variables. We mostly use the following well-known bound on the lower
and the upper tails of the Binomial distribution due to Chernoff
(see \cite{alon2004probabilistic}, \cite{janson2011random}).

\begin{lemma}\label{Che}
Let $X \sim \emph{\text{Bin}}(n,p)$ and let $\mu=\mathbb{E}(X)$.
Then
\begin{itemize}
    \item $\Pr\left[X<(1-a)\mu\right]<e^{-a^2\mu/2}$ for every
    $a>0$;
    \item $\Pr\left[X>(1+a)\mu\right]<e^{-a^2\mu/3}$ for every $0<a<3/2.$
\end{itemize}
\end{lemma}

\begin{remark}\label{hypergeometricChernoff} The conclusions of Lemma \ref{Che} remain the same when $X$
has the hypergeometric distribution (see \cite{janson2011random},
Theorem 2.10).
\end{remark}

\noindent The following is a trivial yet useful bound.
\begin{lemma}\label{Che2}
Let $X \sim \emph{\Bin}(n,p)$ and $k \in \mathbb{N}$.Then the
following holds: $$\Pr(X\geq k) \leq \left(\frac{enp}{k}\right)^k.$$
\end{lemma}
\begin{proof}
$\Pr(X \geq k) \leq \binom{n}{k}p^k \leq
\left(\frac{enp}{k}\right)^k$.
\end{proof}

\subsection{Graph Partitioning}

The next lemma states that one can partition a digraph into subsets which, proportionally, inherit the lower bound on the in- and out-degree.

\begin{lemma} \label{lemma:random-partition-good-degrees} Let
$c, \varepsilon>0$ be constants, $n$ sufficiently large integer and $0<p := p(n) <1$. Suppose that:
\begin{enumerate}[$(i)$]
\item $D$ is a digraph on $n$ vertices,
\item $U\subseteq V(D)$,
\item $k, s_1,\ldots,s_k \in [n]$ are integers such that \\ $s_i \geq \frac{\log^{1.1} n}{p}$ and  $\sum_{i}s_i\leq |U|$.
\end{enumerate}
Then, there exist disjoint subsets $S_1, \ldots, S_k \subseteq U$
such that the following holds for every $1\leq i\leq k$:
\begin{enumerate}[$(a)$]
\item $|S_i|=s_i $, and
\item for every $v\in V(D)$, if  $d_D^{\pm}(v, U) \geq c p |U|$
then
\begin{equation}
d_D^{\pm}(v, S_i) \geq (1-\varepsilon) c p s_i.
\label{eq:claim-deg-b}
\end{equation}
\end{enumerate}
\end{lemma}
\begin{proof} We prove the lemma only for $d^{+}(v, S_i) \geq (1-\varepsilon) c p s_i$ as the proof for
$d^{-}(v, S_i) \geq (1-\varepsilon) c p s_i$ follows in similar
fashion. Let $U = S_1 \cup \ldots \cup S_k \cup Z$ be a partition of
$U$ taken uniformly at random from all partitions for which $S_i =
s_i$ for every $1 \leq i \leq k$ and the ``leftover" set $Z$ is of
size $|Z| = |U| - \sum_{i=1}^k s_i$. Let $W:= \{v \in V(D) \mid
d^{+}(v, U) \geq
 c p |U| \}$ and let $v \in W$ be an arbitrary vertex from $W$.
The number of out-neighbors of $v$ in $S_i$ is hypergeometrical
distributed, thus we have
$$
\mathbb{E}[ d^{+}(v, S_i)] = d^{+}(v, U) \frac{s_i}{|U|} \geq c
\log^{1.1} n.
$$

Using this and Lemma \ref{Che} we obtain the following upper bound
$$
\Pr \left[d^{+}(v,S_i)\leq (1-\varepsilon)d^+(v, U) \frac{s_i}{|U|}
\right] \leq e^{-\varepsilon^2 c \log^{1.1}n / 2}.
$$
The set $W$ has at most $n$ vertices and  the size of each part
$S_i$ is a positive integer, thus the number of parts $k$ is at most
$n$. Taking the union bound over all parts $S_1, \ldots S_k$ and all
vertices in $W$ we get
\begin{align*}
\Pr  \left [ \exists v \in W \; \exists i \in [k], \,
d^{+}(v,S_i)\leq (1-\varepsilon)d^+(v, U) \frac{s_i}{|U|} \right ] 
 \leq n^2 \cdot e^{-\varepsilon^2 c \log^{1.1} n / 2} = o(1),
\end{align*}
which completes the proof.
\end{proof}

\section{Proof of Theorem \ref{thm:main}}
\label{sec:main_result}

In this section we introduce the definition of an $(n, \alpha, p)$-pseudorandom digraph, which will be the main object of study thoughout the paper. In fact, we prove that for $p = \omega(\frac{\log^8 n }n)$ and any positive constant $\alpha$, an $(n, \alpha, p)$-pseudorandom digraph contains a directed Hamiltonian cycle. Using this result the main theorem follows from a fact that after deleting at most $(1/2 - \beta)$ fraction of the edges from each vertex of $\dnp$ the remaining digraph is w.h.p. $(n, \alpha, p)$-pseudorandom, for some positive constant $\alpha < \beta$.

\begin{definition} \label{definition:pseudorandom} A directed graph $D$ on $n$ vertices is called
\emph{$(n,\alpha,p)$-pseudorandom}
if the following holds:
\begin{enumerate}[\bfseries (P1)]
\item for every $v \in V(D)$ we have \label{property:min-degree-all}
$$ d^{\pm}_D(v, V(D)) \ge (1/2 + 2\alpha)n p,$$ 
\item for every subset $X\subseteq V(D)$ of size $|X|\leq \frac{\log^2 n}{p}$, we have \label{property:few-edges-X}
$$e_D(X)\leq |X|\log^{2.1} n,$$
\item for every two disjoint subsets $X,Y\subseteq V(D)$ of sizes
$|X|,|Y|\geq \frac{\log^{1.1} n}{p}$, we have
\label{property:edges-X-Y}
\begin{align*}
&e_D(X,Y)\leq (1 + \alpha/2)|X||Y|p.
\end{align*}
\end{enumerate}
\end{definition}
Intuitively, we require from an $(n, \alpha, p)$-pseudorandom digraph a certain lower bound on the minimum degree and that it contains no dense subgraph. As it turns out, these properties are sufficient  for containing a directed Hamiltonian cycle.

\begin{theorem} \label{thm:main_pseudo}
Let $\alpha > 0$ be a constant and $n$ sufficiently large integer. Then for $p = \omega(\frac{\log^8 n}{n})$, every $(n, \alpha, p)$-pseudorandom digraph is Hamiltonian.
\end{theorem}

Before we give the proof of Theorem \ref{thm:main_pseudo}, we first show how it implies Theorem \ref{thm:main}.

\begin{proof}[Proof of Theorem \ref{thm:main}]

 Let $\beta$, $n$ and $p$ be as stated in the theorem and let $\alpha = \beta / 4 $. By Theorem \ref{thm:main_pseudo}, it is sufficient to prove that $G \sim \dnp$ w.h.p satisfies that $D = G - H$ is $(n,\alpha,p)$-pseudorandom, for every $H \subseteq G$ as given in Definition \ref{def:local}. Using the fact that in $G$ w.h.p.\ $d_G^+(v),d_G^-(v)\in (1\pm o(1))np$ for each $v\in V(G)$, Definition \ref{def:local} implies that we have to show this for all subgraphs $H \subset G$ which satisfy $d^{+}_H(v), d^-_H(v) \leq (1/2-\beta )np$, for every $v \in V(H)$. Let us consider one such subgraph $H$.

First, observe that that for every vertex $v \in D$ we have
\begin{equation}
d^{\pm}_D(v) \ge (1 - o(1) - 1/2 + \beta)np \ge (1/2 + 2\alpha)np, \label{eq:mindeg_main}
\end{equation}
thus the property \PMinDegall{} holds.

For \PXEdges{}, let $X \subseteq V(G)$ be an arbitrary subset of size at
most $\frac{\log^2 n}{p}$. Since $e_G(X) \sim \Bin( 2\binom{|X|}{2}
, p)$, by Lemma \ref{Che2} we have
\[
\Pr\left[ e_G(X) \geq |X| \log^{2.1} n \right ] \leq \left( \frac{e
|X|^2 p}{|X| \log^{2.1}n} \right)^{|X| \log^{2.1}n}.
\]
A union bound over the choices of $X$ shows that the probability that
there exists a subset $X \subseteq V(G)$ of size $x \le \frac{\log^2n}{p}$ such that $e_G(X) \geq |X| \log^{2.1} n$ is at
most
\begin{align*}
\sum_{x\leq \frac{\log^2 n}{p}}  \binom{n}{x} \left(\frac{e x^2p}{ x \log^{2 .1}n} \right)^{x \log^{2.1}n} 
& \leq \sum_{x\leq \frac{\log^2 n}{p}}  \left(n\left(\frac{exp}{\log^{2 .1}n} \right)^{\log^{2.1}n}\right)^x \\
&\leq \sum_{x\leq \frac{\log^2 n}{p}}  \left(n\left(\frac{e}{\log^{0.1}n} \right)^{\log^{2.1}n}\right)^x = o(1).
\end{align*}
Hence,  \PXEdges{} holds in $G$ and therefore in $D\subseteq G$ as well.

The proof of  \PXYEdges{} goes similarly. Consider disjoint subsets
$X,Y\subseteq V(G)$ of size at least $\log^{1.1} n / p$. Then by Lemma \ref{Che} we have
\[
\Pr\left[ e_G(X,Y) > (1+ \alpha/2) |X| |Y| p\right] <
e^{-\Omega\left(|X||Y|p\right)}.
\]
A union bound over the choices for $X$ and $Y$ yields that the
probability that  \PXYEdges{} fails is upper bounded by
\begin{align*}
     \sum_{x,y = \frac{\log^{1.1} n}{p}}^{n}  \binom{n}{x}  \binom{n}{y}  e^{-\Omega\left(xyp\right)}
     \leq \sum_{x, y = \frac{\log^{1.1} n}{p}}^{n} n^x n^y  e^{-\Omega\left(\max \{x,y\} \log^{1.1} n\right)}
    = o(1).
\end{align*}

This proves $G$ is w.h.p such that $D$ is an $(n, \alpha, p)$-pseudorandom digraph, regardless of the choice of $H$, and thus completes the proof.
\end{proof}

Throughout the paper, unless stated otherwise, we always assume that $D$ is a $(n, \alpha , p)$-pseudorandom digraph where $\alpha$ is a positive constant and $p = \omega(\frac{\log^8 n}{n})$. Moreover, with $\{V_1, V_2, V_3, V_4, V_5\}$ we denote the partition of $V(D)$ given by the following claim.

\begin{claim} \label{claim:partition}
There exists a partition $V(D) = \bigcup_{i = 1}^5 V_i$ of the vertices of $D$, such that the following holds:
\begin{enumerate}[\bfseries (Q1)]
\item for every $v \in V(D)$ and every $i \in \{1,2,3,4,5\}$, we have \label{property:min-degree}
$ d^{\pm}_D(v, V_i) \ge (1/2 + \alpha)|V_i| p,$
\item $|V_1| = (1 + o(1))n / \log^3n$ and  $|V_2|, |V_3|, |V_4| \in  \left(\frac{\alpha}{5(1 + 2\alpha)}n, \frac{\alpha}{4(1 + 2\alpha)}n \right)$. \label{property:size}
\end{enumerate}
\end{claim}
\begin{proof}
Let $s_1 = (1 + o(1)) n / \log^3 n$, $s_2, s_3, s_4 \in (\frac{\alpha}{5(1 + 2\alpha)}n, \frac{\alpha}{4(1 + 2\alpha)}n )$ be arbitrarily chosen integers and $s_5 = n - \sum_{i=1}^4 s_i$.
As $s_i \geq \frac{\log^{1.1}n}{p}$ for all $i \in \{1, 2,3, 4, 5\}$, by applying Lemma \ref{lemma:random-partition-good-degrees} with $c = 1/2 + 2 \alpha$, $\varepsilon = \alpha/3$, $V(D)$ (as $U$), $k=5$, $s_1, s_2, s_3, s_4$ and $s_5$ and by using \PMinDegall{} we obtain sets $V_1, V_2, V_3, V_4$ and $V_5$ ($S_1, S_2, S_3, S_4$ and $S_5$ in the Lemma \ref{lemma:random-partition-good-degrees}) such that  $V(D) =  V_1 \cup V_2 \cup V_3 \cup V_4 \cup V_5$ and 
\begin{align*}
d^{\pm}(v, S_i) & \geq (1 - \varepsilon)c s_i p 
 = (1 - \varepsilon) \cdot (1/2 + 2\alpha) s_i p 
 \ge (1/2 + \alpha)s_i p,
\end{align*}
for every $v \in V(D)$ and $i \in \{1,2,3,4,5\}$, as required.
\end{proof}

\subsection{Proof of Theorem \ref{thm:main_pseudo}}

The following two lemmas will serve as our main tool for proving Hamiltonicity of $D$.

\begin{lemma}[\textbf{Absorbing Lemma}] \label{lemma:absorbing}
There exists a directed path $P^*$ with $V(P^*) \subseteq V_2 \cup V_3 \cup V_4$ such that for every $W \subseteq V_1$ there is a directed path $P^*_{W}$ with $V(P^*_W) = V(P^*) \cup W$ and such that $P^*_W$ and $P^*$ have the same endpoints. 
\end{lemma}


The following lemma states that, under certain assumptions, one can find disjoint $\sigma$-walks connecting specified pairs of vertices, for arbitrary $\sigma$ of length $\Omega(\log n).$ The proof of Connecting Lemma is a modification of a beautiful argument by Montgomery \cite{montgomery2014embedding}.

\begin{lemma}[\textbf{Connecting Lemma}] \label{lemma:connecting-lemma} 
Let $\ell$ and $t$ be integers such that $\ell \geq 10 \log n$ and  $t\geq \frac{4\log^2n}p$ and let $\{(a_i, b_i)\}_{i = 1}^t$ be a family of pairs of vertices from $V(D)$ with $a_i \neq a_j$ and $b_i \neq b_j$ for every distinct $i, j \in [t]$. Assume that $K \subseteq V(D) \setminus \bigcup_{i = 1}^t \{a_i, b_i\}$ is such that 
\begin{enumerate}[(i)]
\item $|K| = \omega(t \ell)$, 
\item for every $v \in K$ we have $d^{\pm}(v, K) \geq (\frac{1}{2} + \alpha)p |K|$ and
\item for every $i \in [t]$ we have 
$$d^{\pm}(a_i, K), d^{\pm}(b_i,K) \geq (1/2 + \alpha) p
|K|.$$
\end{enumerate}
Then for every $\sigma \in \{-,
+\}^{\ell}$ there exist $t$ internally disjoint $\sigma$-walks
$P_1,\ldots,P_t$ such that for each $i$, $P_i$ connects $a_i$ to
$b_i$ and $V(P_i)\setminus \{a_i,b_i\} \subseteq K$.
\end{lemma}

With these two lemmas at hand, we are ready to give a proof of Theorem \ref{thm:main_pseudo}.

\begin{proof}[Proof of Theorem \ref{thm:main_pseudo}] 
Let $P^*$ be a path obtained from Lemma \ref{lemma:absorbing} and let $U := (V_2 \cup V_3 \cup V_4 \cup V_5) \setminus V(P^*)$. We first show that there exists a family $\{Q_1, \ldots, Q_{t'}\}$ of $t' \in [n / \log^5n, 2 n / \log^5 n]$ vertex-disjoint directed paths  such that $U = \bigcup_{i=1}^{t'} V(Q_i)$. It follows from property \PSize{} that $|V_5| \ge (1 - \frac{\alpha}{1 + 2\alpha})n$. Furthermore, property \PMinDeg{}  and $U \subseteq V_2 \cup V_3 \cup V_4 \cup V_5$ imply
\begin{align}
d^{\pm}(v, U) & \ge d^\pm(v, V_5) \ge (1/2 + \alpha)p|V_5| 
 \ge (1/2 + \alpha)p  \frac{1 + \alpha}{1 + 2\alpha}|U| 
 = (1/2 + \alpha/2)p|U|, \label{eq:d_U}
\end{align}
for every $v \in U$. Applying Lemma \ref{lemma:random-partition-good-degrees} with $c = 1/2 + \alpha/2$, $\varepsilon$ such that $(1 - \varepsilon)c > 1/2 + \alpha/4$, $U$, $k = \lfloor \log^5 n |U| / n \rfloor$ and $s_i = \lfloor n / \log^5 n \rfloor$ for every $i \in [k]$, together with \eqref{eq:d_U}, we obtain disjoint subsets $S_1, \ldots, S_k \subseteq U$ such that
\begin{equation}
d^{\pm}(v, S_i) \ge (1 - \varepsilon)cp |S_i| \ge (1/2 + \alpha')p|S_i|, \label{eq:alpha_prim}
\end{equation}
for some constant $\alpha' > \alpha / 4$, every $i \in [k]$ and $v \in U$. We now use the following claim, whose proof we defer to the end of the subsection.

\begin{claim} \label{claim:perfect_match}
For every $i \in [k-1]$, there exists a perfect matching from $S_i$ to $S_{i+1}$.
\end{claim}

Observe that such matchings induce $s := \lfloor n / \log^5 n \rfloor$ vertex-disjoint directed paths $\{Q_1, \ldots, Q_{s}\}$, each of length $k$, such that $\bigcup_{i = 1}^{s} V(Q_i) = \bigcup_{i = 1}^{k} S_i$. On the other hand, by taking each vertex in $U \setminus \bigcup_{i = 1}^{k} S_i$ to be a $0$-length path, we obtain at most $s$ additional paths $\{Q_{s + 1}, \ldots, Q_{t'}\}$. Note that the family $\{Q_1, \ldots, Q_{t'}\}$ satisfies the desired properties.

As a final step, we find a cycle $C$ in $D$ which contains paths $Q_1, \ldots, Q_{t'}, P^*$ and maybe some vertices from $V_1$. Using Lemma \ref{lemma:absorbing}, we can absorb the remaining vertices from $V_1$ and obtain a Hamiltonian cycle. We now make this more precise.

For $i \in [t']$, let us denote with $a_i$ and $b_i$ the first and the last vertex on the path $Q_i$. Furthermore, let $a_{t'+1}$ and $b_{t'+1}$ be the first and the last vertex of the path $P^*$. Applying Lemma \ref{lemma:connecting-lemma} with $\ell = 10 \log n$, $t = t' + 1$, the family of pairs $\{(b_{i}, a_{i+1})\}_{i = 1}^{t'} \cup \{(b_{t'+1}, a_1)\}$ and $V_1$ (as $K$), we obtain vertex disjoint directed paths $P_1, \ldots, P_{t' + 1}$. Observe that $Q_1, P_1, Q_2, \ldots, Q_{t'}, P_{t'}, P^*, P_{t'+1}$ forms a directed cycle $C$ with $V_2 \cup V_3 \cup V_4 \cup V_5\subseteq V(C)$.
By Lemma \ref{lemma:absorbing} there is a path $P^*_{V_1 \setminus V(C)}$ with the same endpoints as $P^*$ and such that $V(P^*_{V_1 \setminus V(C)}) = V(P^*) \cup (V_1 \setminus V(C))$.
 As $C$ contains the path $P^*$, we can replace $P^*$ with $P^*_{V_1 \setminus V(C)}$, thus obtaining a Hamiltonian cycle.
\end{proof}

\begin{proof}[Proof of Claim \ref{claim:perfect_match}]

We use the following theorem which is equivalent to Hall's condition (see \cite{West}): There exists a perfect matching from $S_i$ to $S_{i+1}$ if
and only if for every subset $X\subseteq S_i$ of size $|X|\leq |S_i|/2$
we have $|N^+(X,S_{i+1})|\geq |X|$ and for every subset $Y\subseteq S_{i+1}$ of
size $|Y|\leq |S_{i+1}|/2$ we have $|N^-(Y,S_i)|\geq |Y|$. Assuming the opposite,
without loss of generality there exists a subset $X\subseteq S_i$ of
size $|X|\leq |S_i|/2$ for which $N^+(X,S_{i+1})$ is contained in a set $Y$
of size exactly $|X|-1$. We distinguish between two cases:
\begin{enumerate}[(a)]
\item $|X|\leq \log^2n/(2p)$. In this case we have that
$e_D(X,S_{i+1}\setminus Y)=0$, and therefore, using \eqref{eq:alpha_prim} we obtain that 
\begin{align*}
e_D(X \cup Y) & \ge e_D(X,Y) = e_D(X, S_{i+1}) 
 \geq (1/2+\alpha')|X||S_{i+1}|p 
 =\omega(|X|\log^{2.1} n),
\end{align*}
which contradicts \PXEdges{} (here we use the fact that $|S_{i+1}|p=\omega(\log^{2.1} n)$).
\item $\log^{1.1}n/p<|X|\leq |S_i|/2$. In this case we have
\begin{align*}
e_D(X,Y) & \geq (1/2+\alpha')|X||S_{i+1}|p 
 = \frac{1+2\alpha'}{2} |X||S_{i+1}|p 
 > (1+\alpha/2)|X||Y|p,
\end{align*}
which contradicts \PXYEdges{} (here we use the fact that $\alpha' > \alpha/4$).
\end{enumerate}
The same argument can be applied to a subset $Y \subseteq S_{i+1}$ of size $|Y| \le |S_{i+1}|/2$. This completes the proof.
\end{proof}

\section{Proof of the Connecting Lemma}
\label{sec:connecting_lemma}

In  this section we prove Lemma \ref{lemma:connecting-lemma}. The lemma states that if we have a list of $t$ pairs of vertices and a set $K$ of size $\omega(t \ell)$ (where $\ell \geq 10\log n$) which ``behaves" like a  random subset of $V(D)$ then for any $\sigma \in \{+,-\}^{\ell}$ we can connect each pair of vertices via $t$ disjoint  $\sigma$-walks of length $\ell$ by using only vertices from $K$.
The proof is  obtained by adopting a clever argument due to Richard Montgomery \cite{montgomery2014embedding} into the setting of resilience. In order to do so, we had to repeat the whole argument.

\subsection{Expansion properties and $\sigma$-neighborhoods}

We start with a lemma which says that for any two (not too small) sets $X, Y \subseteq V(D)$, such that all vertices $x \in X$ have a large degree in $Y$, $X$ expands to more than a half of vertices in $Y$.
\begin{lemma}
\label{lemma:big-expansion} 
Let $X, Y \subseteq V(D)$ be two (not necessarily
disjoint) subsets such that  $|X| = \lfloor \frac{\log^{2} n}{2p} \rfloor$, $|Y|
\geq \frac{6 \log^{2.1 }n}{\alpha p}$ and for each $x \in X$
$$ d^{\pm}(x, Y) \ge (1 / 2 + \alpha / 2) p |Y|.$$
Then  $|N^+(X, Y) |, |N^-(X, Y) |  \geq (1/2 + \alpha / 20)|Y|$.
\end{lemma}

\begin{proof}
We only prove that $|N^+(X, Y)| \geq (1/2 + \alpha / 20)|Y|$ as  the bound on $|N^-(X, Y)|$ can be proven analogously.
From property \PXEdges{}  we have $e(X ) \leq
|X| \log^{2.1} n$.
Now, denote $S_X := N^+(X, Y) \setminus
X$. Using the previous inequality together with
 $d^{\pm}(x,Y)\geq (1/2 + \alpha / 2)p |Y|$ for every $x \in X$ and  $|Y|p  \geq  6 \log^{2.1 }n / \alpha$, we obtain
 \begin{align}
e(X, S_X) & \geq  e(X, Y) - e(X) 
  \geq (1/2 + \alpha/2 ) p |X| |Y| - |X| \log^{2.1} n 
  \geq (1/2 +  \alpha / 3 ) p |X| |Y| . \label{eq:e_x_sx}
 \end{align}

Now let us assume  $|S_X| < \frac{\log^2 n}{2p}$.  We then know that  $|X \cup S_X| < \frac{\log^2 n}{p}$. Therefore we can apply property
\PXEdges{} to the set $X \cup S_X$ and conclude $e(X \cup S_X)
\leq |X \cup S_X| \log^{2.1} n
\leq 2
|X| \log^{2.1} n$. On the other hand, from Equation
\eqref{eq:e_x_sx} and the bound on the size of $Y$ we have
\begin{align*}
e(X \cup S_X)  \geq  e(X ,S_X) \geq (1/2 + \alpha /3 ) p |X|
|Y| 
  \geq (2 + 3/\alpha) |X| \log^{2.1} n ,
 \end{align*}
 which is a contradiction.

Next, we assume $\frac{\log^2 n}{2p} \leq |S_X | \leq
 (1/2 + \alpha / 20) |Y|$.
 Using property \PXYEdges{} we obtain
 $$e(X,S_X)  \leq (1 + \alpha / 2 ) (1/2 + \alpha/ 20)p |X| |Y|.$$
  Now, combining the previous inequality with \eqref{eq:e_x_sx} we conclude
\begin{align}
(1/2 + \alpha /3 ) p |X| |Y| & \leq e(X,S_X)
 \leq  (1 + \alpha / 2 ) p |X| (1/2 + \alpha /20) |Y|. \label{eq:big_expan}
\end{align}
However, by easy calculation one can check that $(1 + \alpha / 2)(1/2 + \alpha / 20) < (1/2 + \alpha /3)$, thus Equation \eqref{eq:big_expan} gives a contradiction. Therefore we have $|N^+(X, Y)| \ge |S_X| \ge (1/2 + \alpha/20)|Y|$, as required.
\end{proof}

Next, we introduce the notion of $\sigma$-neighborhood.
For given sets $A, B \subseteq V(D)$, integer $\ell$ and  $\sigma \in \{+,-\}^{\ell}$,
we define $N^{\sigma}(A, B)$ as follows,
\begin{align*}
 N^{\sigma}(A, B) := \{ & x \in B \mid \exists a_x \in A \text{ and } 
  \text{a } \sigma\text{-walk } P \text{ connecting } a_x \text{ to } x \text{ and } 
   V(P) \setminus \{a_x\} \subseteq B\}.
 \end{align*}


In the following lemma we show  that for two subsets $X,Y \subseteq V(D)$, such that $X$ and $Y$ have good expansion properties, we can find a vertex $x \in X$ which can reach more than a half of the vertices from $Y$ via $\sigma$-walks.

\begin{lemma}
\label{lemma:expand-1-out-of-many} 
Let $\gamma \in (0,1)$ be a constant and let  $\ell$ be an integer such
that $\ell \geq 2 \log n $. Suppose that $X, Y \subseteq V(D)$ are two disjoint subsets of vertices such that the following holds:
\begin{enumerate}[$(i)$]
    \item $|Y| \geq \frac{6\ell }{\gamma}\cdot \lceil \frac{\log^2 n}{ p} \rceil$,
    \item $|N^+(X, Y)|,|N^-(X, Y)| \geq \frac{2 \log^2 n}{p}$ and 
    \item for every subset $S \subseteq Y $ of size $S \geq \frac{\log^2 n}{p}$ we have
    $$
    |N^+(S, Y)|,|N^-(S, Y)| \geq (1/2 + \gamma) |Y|.
    $$
\end{enumerate}
Then for any $\sigma \in \{+, -\}^{\ell}$ there exists a vertex $x \in X$ such that 
$$|N^{\sigma}(x, Y)| \geq (1/2 + \gamma/2)|Y|.$$
\end{lemma}

\begin{proof}
Recall that $\sigma^i = (\sigma(1), \ldots, \sigma(i))$. We first show that there exists a vertex $x \in X$ such that $|N^{\sigma^{\ell - 1}}(x, Y)| \ge \frac{2 \log^2 n }{p}$. In order to do so, we make use of the following claim.

\begin{claim}
\label{claim:sigmai_ip1} Let $i < \ell$ be an integer and $A \subseteq X$ such that $|N^{\sigma^i}(A, Y)| \geq
\frac{2 \log^2 n}{p}$. Then there exists a subset $A' \subseteq A$
such that $|A'| \leq \lceil |A| / 2 \rceil$ and
$$
|N^{\sigma^{i+1}}(A',  Y)| \geq \frac{2 \log^2 n}{p}.
$$
\end{claim}
By assumption $(ii)$ we have $|N^{\sigma(1)}(X, Y)| \ge 2\log^2n/p$, thus applying Claim \ref{claim:sigmai_ip1} repeatedly $\ell - 2$ times we obtain a set $X' \subseteq X$ such that $|X'| \leq \lceil |X|/2^{\ell - 2} \rceil$ and $|N^{\sigma^i}(X', Y)|
\geq \frac{2\log^2 n}{p}$. Since $|X|\leq n$ and $\ell-2 \geq \log
n$, it follows that $|X| / 2^{\ell - 2} \leq 1$ and therefore $|X'| = 1$. Hence, there exists $x \in X $ such that
$|N^{\sigma^{\ell-1}}(x, Y)| \geq \frac{2 \log^2 n}{p}$.

Let now $M \subseteq N^{\sigma^{\ell - 1}}(x, Y)$ be a subset of size $|M| =
\lceil \frac{\log^2 n}{p} \rceil$ and note that, by definition, for
each $w \in M$ there exists a $\sigma^{\ell - 1}$-walk $P_w$ connecting $x$ to $w$
with $V(P_w)\setminus\{x\} \subseteq Y$. Let $V^* :=
\left(\bigcup_{w \in M} V(P_w)\right)\setminus \{x\}$. Using
assumptions $(ii)$ and $(iii)$ we have
\begin{align*}
|N^{\sigma(\ell)}(M, Y  \setminus V^*) | & \geq (1/2 + \gamma ) |Y| - \ell |M|  
 \geq (1/2 + \gamma/2  )  |Y|.
\end{align*}
Observe that  $N^{\sigma(\ell)}(M, Y \setminus  V^*) \subseteq N^{\sigma}(x, Y)$ and therefore
$|N^{\sigma}(x, Y)| \geq (1/2 + \gamma/2  )  |Y|$.

In order to complete the proof it remains to prove Claim
\ref{claim:sigmai_ip1}.

\begin{proof}[Proof of Claim \ref{claim:sigmai_ip1}]

First, note that there exists a subset $A' \subseteq A$ such that $ |A'| \leq \lceil |A| / 2 \rceil$ and
$|N^{\sigma^i}(A', Y)| \geq \frac{\log^2 n}{p}$.
Indeed, this is true as otherwise taking an arbitrary partition of  the set $A = S \cup T$, such that $|S|, |T| \leq \lceil |A| / 2 \rceil$, yields
$$
|N^{\sigma^i}(A ,Y)| \leq |N^{\sigma^i}(S, Y)|  +  |N^{\sigma^i}(T, Y)|  < \frac{2 \log^2 n}{p},
$$
which contradicts the assumption that $|N^{\sigma^i}(A, Y)| \geq \frac{2 \log^2 n}{p}$.

Let $H \subseteq N^{\sigma^{i}}(A',Y)$ be an arbitrary subset of size  $|H| = \lceil \frac{\log^2 n}{p} \rceil$.
Using  assumption $(iii)$  we have
$
|N^{\sigma(i+1)}(H, Y)| \geq (1/2 + \gamma) |Y|.$
We know that  for each $v \in H$ there exist a $\sigma^i$-walk $P_v$ connecting
a vertex from $A'$ to the vertex $v$. Let us denote $V^* := \cup_{v \in H} V(P_v)$.
Using the upper bound on $i$ we have  $|V^*| \leq \ell |H|$ and thus
$$
|N^{\sigma(i+1)}(H, Y \setminus V^* )| \geq (1/2 + \gamma) |Y| - \ell |H|  \geq
 \frac{2\log^2 n}{p},
$$
where the second inequality follows from assumption $(i)$.
Finally, observe that   $N^{\sigma(i+1)}(H, Y \setminus V^* ) \subseteq N^{\sigma^{i+1}}(A', Y)$ and hence we have
$|N^{\sigma^{i+1}}(A', Y)| \geq \frac{2\log^2 n}p$.
\end{proof}
This completes the proof of the lemma. \end{proof}

\subsection{The proof}

The following lemma is an approximate version of the Connecting Lemma and it is used as the main building block in the proof of the Connecting Lemma. 
 Namely, the lemma states that for a given set of pairs $\{(a_i, b_i)\}_{i = 1}^t$, sets $R_A, R_B$ with good expansion properties we can connect half of the pairs via long $\sigma$-walks using only vertices from $R_A \cup R_B$.

\begin{lemma}
\label{lemma:connect-half-pairs} 
Let $\gamma \in (0,1)$ be a constant, let $\ell$ and $t$ be integers such that $\ell \geq 5 \log n$ and  $t\geq \frac{4\log^2n}p$ and let $\{(a_i, b_i)\}_{i = 1}^t$ be a family of pairs of vertices from $V(D)$ with $a_i \neq a_j$ and $b_i \neq b_j$ for every distinct $i, j \in [t]$.
Furthermore, let $R_A,R_B\subseteq  V(D) \setminus  (\bigcup_{i=1}^t\{a_i, b_i\})$ be disjoint subsets such that the following holds:
\begin{enumerate} [$(i)$]
\item $|R_A|, |R_B| \geq  12 t \ell / \gamma$ and 
\item for $X \in \{A, B\}$ and for every set $S \subseteq  R_A \cup R_B \cup \bigcup_{i=1}^t\{a_i, b_i\}$ of size at least
$\frac{\log^2 n}p$ we have $$|N^+(S,R_X)|,|N^-(S,R_X)|\geq (1/2 + \gamma)|R_X|.$$
\end{enumerate}
Then for any $\sigma \in \{+, -\}^{\ell}$ there exists a subset of indices $\mathcal I \subseteq [t]$ of
size $s:= \lfloor t / 2 \rfloor $ and $s$  internally disjoint
$\sigma$-walks $P_i$ which connect $a_i$ to $b_i$ (where $i \in
\mathcal I$), and such that
$V(P_{i})\setminus\{a_{i},b_{i}\}\subseteq R_A\cup R_B$.
\end{lemma}
\begin{proof}
We prove the existence of set $\mathcal I$ of size $\lfloor t/  2 \rfloor$ and the required $\sigma$-walks by induction.
Assume that there exists $\mathcal I = \{i_1, \ldots, i_{s'}\} \subseteq [t]$ with  $s'< \lfloor t / 2
\rfloor$, and $s'$ vertex-disjoint
$\sigma$-walks $P_i$, connecting $a_i$ to $b_i$, where $i\in
\mathcal I$. Let us define
\begin{gather*}
R'_A :=R_A \setminus \cup_{i \in \mathcal I} V(P_{i}),\quad R'_B := R_B
\setminus \cup_{i \in \mathcal I} V(P_i) 
\quad \text{and} \quad
\mathcal{I}' := [t] \setminus \mathcal I.
\end{gather*}
Next, we show how to find a $\sigma$-walk $P$ connecting some
$a_i$ to $b_i$ where $i \in \mathcal I'$ such that
$V(P)\setminus\{a_i,b_i\}\subseteq R'_A\cup R'_B$. Let $h_A$ and
$h_B$ be two integers such that $h_A,h_B\geq 2\log n$ and
$h_A+h_B+1=\ell$, and consider $\sigma^{h_A}$ and
$\bar{\sigma}^{h_B}$ (recall that
$\sigma^{h_A}=(\sigma(1),\ldots,\sigma(h_A))$ and
$\bar{\sigma}^{h_B}=(\bar\sigma(\ell),\ldots,\bar\sigma(\ell-h_B+1))$). We make use of the following claim.

\begin{claim}\label{claim:halfaregood}
There exists an index $i\in \mathcal I'$ for which the following
holds:
\begin{gather*}
|N^{\sigma^{h_A}}(a_i,R'_A)|\geq (1/2 + \gamma/4)|R'_A|\\ 
\text{and}\quad  \\
|N^{\bar{\sigma}^{h_B}}(b_i,R'_B)|\geq (1/2 + \gamma/4)|R'_B|.
\end{gather*}
\end{claim}

Before we prove this claim we show how to finish the proof of the
lemma. Let $(a_i,b_i)$  be a pair of vertices with index  obtained by Claim
\ref{claim:halfaregood}. For $S:=N^{\sigma^{h_A}}(a_i,R'_A)$
we have that $|S|\geq (1/2 + \gamma/4)|R'_A|$. As $R'$ is obtained by removing vertices of at most $|\mathcal I|$ many 
$\sigma$-walks we know 
$
 |R'_A| \geq |R_A| - t \ell \geq 11 t \ell, 
$
and thus  $|S| \geq \log^2 n / p$.
By assumption $(ii)$ we have that $N^{\sigma(h_A + 1)}(S, R_B) \geq (1/ 2 + \gamma) |R_B|$ and consequently
\begin{align}
N^{\sigma(h_A + 1)}(S, R'_B) & \geq (1/ 2 + \gamma) |R_B| - t \ell
\stackrel{(i)}{\geq} (1/ 2 + \gamma/2 ) |R_B|
 \ge 
(1/ 2 + \gamma/ 2) |R'_B|. \label{eq:half_r'b}
\end{align}
On the other hand we know from Claim \ref{claim:halfaregood} that 
$|N^{\bar{\sigma}^{h_B}}(b_i,R'_B)|\geq
(1/2 + \gamma/4)|R'_B|$. This implies together with Equation \eqref{eq:half_r'b} that there exist $v \in N^{\sigma^{h_A}}(a_i,R'_A)$ and $w \in  N^{\bar{\sigma}^{h_B}}(b_i,R'_B)$ such that 
$$
w \in N^{\sigma(h_A +1)}(v).
$$
Therefore we can construct a $\sigma$-walk $P$ connecting $a_i$ to $b_i$ such that $P$ is vertex disjoint from all previous $\sigma$-walks.
Now it only remains to prove Claim \ref{claim:halfaregood}.

\begin{proof}[Proof of Claim \ref{claim:halfaregood}] 

The idea of the proof is to repeatedly apply Lemma \ref{lemma:expand-1-out-of-many}. First, we apply Lemma \ref{lemma:expand-1-out-of-many}  to $X := \bigcup_{i \in \mathcal I'} \{a_i\}$, $Y := R'_A$, $\gamma / 2$ (as $\gamma$) and obtain a vertex $v_1 \in X$ such that $|N^{\sigma^{h_A}}(a, R'_A)| \geq (1 /2 + \gamma/4) |R'_A|$. Next, we apply the lemma again but now to $X := X \setminus \{v_1\}$ instead (with other parameters unchanged) and obtain  $v_2 \in X \setminus \{v_1\}$. After $k$ steps of this procedure we obtain 
vertices $\{v_1, \ldots, v_k\}$ with the property $|N^{\sigma^{h_A}}(v_i, R'_A)| \geq (1 /2 + \gamma/4) |R'_A|$, for all $1 \leq i \leq k$.

Let us now argue that we can indeed apply Lemma \ref{lemma:expand-1-out-of-many} and, moreover, estimate the number of steps $k$. Note that the condition $(i)$ from Lemma \ref{lemma:expand-1-out-of-many} is satisfied as
$$|R'_A| \geq  12 t \ell / \gamma \ge \frac{20 \log^3 n}{\gamma p}.$$ 
On the other hand, using property $(ii)$ of Lemma \ref{lemma:connect-half-pairs} and the fact that $|R_A| - |R'_A| \leq t \ell$ we get
\begin{align*}
|N^-(S,R'_A)|, |N^+(S,R'_A)| & \geq (1/2 + \gamma) |R_A| - t \ell 
 \geq (1/2 + \gamma/2) |R'_A|,
\end{align*}
for any $S \subseteq R'_A \cup \bigcup_{i \in \mathcal I'}\{a_i\}$ of size at least $\frac{\log^2 n}{p}$.
Therefore, $X = \bigcup_{i \in \mathcal{I}'}\{a_i\} \setminus \{v_1, \ldots, v_i\}$ satisfies assumption $(ii)$ of Lemma \ref{lemma:expand-1-out-of-many} as long as $|\mathcal{I}'| - i > \log^2n / p$. This implies
that we can iterate the process for at least $k \ge |\mathcal I'| / 2  + 1$ steps as $|\mathcal I'| > t / 2 \geq \frac{2 \log^2 n }{p}$. Thus, we obtain $V_A := \{v_1, \ldots, v_{k}\}$ with $v_j \in \{a_i\}_{i \in \mathcal I'}$ and
$$|N^{\sigma^{h_A}}(v_j, R'_A)| \geq (1/2 + \gamma/4) |R'_A|$$
for all $j \in [k]$.

By using the analogous argument with $\{b_i\}_{i \in \mathcal I'} $ and $R'_B$  we  obtain $V_B := \{w_1, \ldots, w_{k}\}$ such that $k > |\mathcal I'| /2 $ and $w_j \in \{b_i\}_{i \in \mathcal I'} $ with the property 
$$|N^{\bar{\sigma}^{h_B}}(w_j, R'_B)| \geq (1/2 + \gamma/4) |R'_B|$$
for all $j \in [k]$. Therefore, there must exist $i \in \mathcal I'$ such that $a_i \in V_A$ and $b_i \in V_B$, as required by the claim.
\end{proof}

This finishes the proof of Lemma \ref{lemma:connect-half-pairs}.
\end{proof}

Before proving Connecting Lemma we need to introduce the following definitions.

\begin{definition}
Let $T$ be a rooted tree with edges oriented arbitrarily. Let
$\mathcal L(T)$ denote the set of leaves of $T$ and let $\sigma\in
\{+,-\}^\ell$ for some integer $\ell$. We say that $T$ is a
\emph{$\sigma$-tree} if for each $v \in \mathcal L(T)$ the unique
path from the root of $T$ to $v$ is a $\sigma$-walk.
\end{definition}

\begin{definition}
Let $\tau \in \{+,-\}$ and let $X,Y \subseteq V(D)$ be two disjoint sets. We say that there is a 
$(2, \tau)$-matching between $X$ and $Y$ that saturates $X$ if
for each $x \in X$ there are two distinct vertices $y^1_x, y^2_x \in Y$ such that $\{y^1_x, y^2_x\} \in N^{\tau}(x)$ and $\{y^1_x, y^2_x\} \cap \{y^1_{x'}, y^2_{x'}\} = \emptyset$ for $x \neq x'$.
\end{definition}

We are finally ready to prove the main lemma of this section.

\begin{proof}[Proof of Lemma \ref{lemma:connecting-lemma}]

Let $\sigma$ be an arbitrary element of $\{+,-\}^{\ell}$ and 
let $\varepsilon>0$ be a sufficiently small constant (to be
determined later). Throughout the proof we make use of the
following parameters:
\begin{align*}
 h &= \max\left\{\frac{6\log^{2.2} n}{p},2t \right\}, \\
 m &= \lceil \log_2 t \rceil +1, \\
 s_i &=h \quad \text{for} \quad 1 \leq i\leq 2m, \\
 s_{2m+1} &=s_{2m+2}=\frac{|K|}{4}, \\
 k&=2m+2.
\end{align*}
Applying Lemma
\ref{lemma:random-partition-good-degrees} to $s_1, s_2, \ldots,s_k$,
$\varepsilon$, $1/2 + \alpha$ (as $\gamma$), $p$, $K$ (as $U$) and $D$ we
obtain disjoint subsets $S_1, \ldots, S_k \subseteq K$, such that
for every $1\leq i\leq k$ the following holds:
\begin{enumerate}[$(a)$]
\item $|S_i|=s_i $, and
\item for every $v\in V(D)$, if  $d_D^{\pm}(v, K) \geq (1/2 + \alpha)
p |K|$, then
\begin{equation}
d_D^{\pm}(v, S_i) \geq (1-\varepsilon)  (1/2 + \alpha) p s_i.
\label{eq:consequence-b}
\end{equation}
\end{enumerate}
Using \eqref{eq:consequence-b}, properties $(ii)$ and $(iii)$ and the fact that $\varepsilon$ is
sufficiently small, we obtain that for any $v \in \{a_i \cup b_i\}_{i =1}^t \cup K$ and for any
set $S_i$ the following holds:
\begin{equation}
d_D^{\pm} (v, S_i) \geq (1/2 + \alpha/2)  p s_i.
\label{eq:good-deg-in-parts}
\end{equation}

For simplicity of presentation, let us denote $A_0:=\bigcup_{i=1}^i \{a_i\}$, $B_0:=\bigcup_{i=1}^i \{b_i\}$,
$A_i:=S_i$ for each $1\leq i\leq m$, $B_i:=S_{m+i}$ for each $1\leq
i\leq m$, $R_A:=S_{2m+1}$ and $R_B:=S_{2m+2}$.

We first describe, informally, the strategy for finding $\sigma$-walks. 
In a first step, we apply Lemma~\ref{lemma:connect-half-pairs} to find $t/2$ $\sigma$-walks between vertices in $A_0$ and $B_0$. Then we find a $(2, \sigma(1))$-matching  between the leftovers in $A_0$ and the vertices in $A_1$ and a $(2, \bar\sigma(\ell))$-matching between the leftovers in $B_0$ and the vertices in $B_1$. Let $A'_1$ denote the set of vertices that are matched to a leftover of $A_0$ and analogously define $B_1'$.
Observe that $|A_1'|=|B_1'|\geq t$ and therefore one can apply Lemma~\ref{lemma:connect-half-pairs} to find $|A'_1|/2$ vertex disjoint $\kappa$-walks between vertices in $A'_1$ and $B'_1$, where $\kappa:=(\sigma(2), \ldots, \sigma(\ell-1))$. Note that extending the walks with the matchings yields $\sigma$-walks between at least $t/4$ leftovers of $A_0$ and the corresponding leftovers of $B_0$. By iteratively continuing this process for roughly $\log t$ steps, we construct all the desired walks.

Before proceeding with the description of the procedure, we define the following invariant which we maintain in every step $0 \leq s \leq m$:
\begin{enumerate}[$(X1)$]
\item $\mathcal I_s \subseteq [t]$ for which
\begin{enumerate}
\item if $s < m $ then $|\mathcal I_s|= t - \lceil t / 2^s \rceil$
\item if $s = m $ then $|\mathcal I_s|= t$
\end{enumerate}
\item $\mathcal P_s = \{P_i\}_{i \in \mathcal I_s}$ is a collection of vertex-disjoint $\sigma$-walks such that $P_i$ connects $a_i$ to $b_i$ and
$V(P_i) \setminus \{a_i, b_i\} \subseteq \left(\bigcup_{k=1}^s (A_k
\cup B_k)\right) \cup R_A \cup R_B$, for all $i \in \mathcal I_s$
\item $\mathcal T_A^s = \{T^i_A \}_{i \in [t] \setminus \mathcal I_s}$ is a collection of $\sigma^s$-trees
and $\mathcal T_B^s = \{T^i_B \}_{i \in [t] \setminus \mathcal I_s}$
is a collection of $\bar \sigma^s$-trees
 such that for each $i \in [t] \setminus \mathcal I_s$
\begin{enumerate}
\item $T_A^i$ is rooted at $a_i$, $\mathcal L(T_A^i) \subseteq A_s$, $|\mathcal L(T_A^i)| = 2^s$ and $V(T_A^i) \setminus \{a_i\} \subseteq \cup_{j=1}^s A_j$
\item $T_B^i$ is rooted at $b_i$, $\mathcal L(T_B^i) \subseteq B_s$, $|\mathcal L(T_B^i)| = 2^s$ and $V(T_B^i) \setminus \{b_i\} \subseteq \cup_{j=1}^s B_j$
\end{enumerate}
\item for any two $T_1, T_2 \in \mathcal T_A^s \cup \mathcal T_B^s$
rooted at $v_1$ and $v_2$ we have $(V(T_1) \setminus v_1) \cap (V(T_2) \setminus v_2) = \emptyset$
\item for any two $T \in \mathcal T_A^s \cup \mathcal T_B^s$ and any $P \in \mathcal P_s$ we have
$V(T) \cap V(P) = \emptyset$
\end{enumerate}

The set $\mathcal I_s$ represents a set of indices of pairs which are connected by a $\sigma$-walk up to step $s$. 
The collection $\mathcal P_s$ contains $\sigma$-walks created up to step $s$ between pairs with indices in $I_s$ and the $\mathcal T^s_A$ and $\mathcal T^s_B$ are collections of trees for each element of a pair  not connected by a $\sigma$-walk up to step $s$.

First, the invariant clearly holds for $ s= 0$, $\mathcal I_0 = \emptyset$, $\mathcal T^0_A = A_0$,
$\mathcal T^0_B = B_0$ and $\mathcal P_0 = \emptyset$.
Suppose that  the invariant holds for some $s$ such that $s < m$, we will show how to extend it to $s+1$.
Denote $A'_s = \bigcup_{T \in \mathcal T_A^s} \mathcal L(T)$ and
$B'_s = \bigcup_{T \in \mathcal T_B^s} \mathcal L(T)$. 
Let $\{a'_i, b'_i\}_{i=1}^r$ be a perfect matching between vertices  of $A'_s$ and $B'_s$ with the following property:
for every $1 \leq i \leq r$ there is $j \in [t]$ such that $a'_i$ and $b'_i$ are leaves of trees rooted at $a_j$ and $b_j$.
Using $(X1)$ and $(X3)$ we obtain that $r = 2^s (t - |\mathcal I_s|) \geq 2^s  \lceil t /2^{s} \rceil \geq t$.
Next, let $R_A' = R_A \setminus \cup_{i \in \mathcal I_s} V(P_i)$
and let $R_B' = R_B \setminus \cup_{i \in \mathcal I_s} V(P_i)$. Using Claim \ref{claim:3} below it follows that for every $X\in\{A,B\}$ and every subset $S\subseteq
K $ such that  $|S| \geq \frac{\log^2 n}p$ we have
\begin{align*}
 |N^+(S,R'_X)|,|N^-(S,R'_X)|\geq (1/2+ \alpha /40)|R'_X|.
\end{align*}
Therefore, we can apply Lemma \ref{lemma:connect-half-pairs} to
$\alpha /40$ (as $\gamma$), the family of pairs $\{(a'_i, b'_i)\}_{i=1}^r$, $R'_A$ (as $R_A$), $R'_B$ (as $R_B$) and
$\kappa :=(\sigma(s+1),\ldots,\sigma(\ell-s))$, and obtain the
following: a set of indices $\mathcal J \subseteq [r]$ of size
$|\mathcal J|= \lfloor r/2 \rfloor = 2^{s-1}(t - |\mathcal I_s|)$ and a collection of vertex-disjoint $\kappa$-walks $W_{i}$, such
that a path $W_{i}$ connects $a'_i$ to
$b'_i$ and $V(W_{i})\setminus \{a'_i,b'_i\}\subseteq R'_A\cup R'_B$, for each $i \in \mathcal J$.

Let us pick a subset $\mathcal J' \subseteq \mathcal J$ of size
 $\lfloor |\mathcal J| /2^s \rfloor = \lfloor (t - \mathcal |I_s|) /2 \rfloor$ such that
each tree from $\mathcal T_B^s \cup \mathcal T_A^s$ has at most one leaf indexed with some $i' \in \mathcal J'$. Note that this is possible since each tree has $2^s$ leaves. For technical reasons, when $s = m  -1$ we pick $\mathcal J'$ of size  $|\mathcal J'| = 1$.
 Now, for any $j \in \mathcal J'$ let $Q^1_j$ be the unique
$\sigma^s$-walk in the tree $T_1 \in \mathcal T^s_A$ containing  $a'_j$ which connects the root of $T_1$ to $a'_j$. Similarly, let $Q^2_j$ be the unique
$\bar \sigma^s$-walk in the tree $T_2 \in \mathcal T^s_B$ containing  $b'_j$ which connects the root of $T_2$ to $b'_j$.
We know that $Q^1_j$ and $Q^2_j$ start in vertices with the same index from $A_0 $ and $B_0$, by the definition of matching between $A'_s$ and $B'_s$. Combining the paths $Q^1_j, Q^2_j$ and $W_j$ we obtain a $\sigma$-walk
  $P_{i_j}$ which connects vertex $a_{i_j}$ to $b_{i_j}$ for some $i_j \in [t] \setminus \mathcal I_s$.
  We define $\mathcal P_{s+1} = \mathcal P_s \cup (\cup_{j \in \mathcal J'} P_{i_j})$ and $\mathcal I_{s+1} = \mathcal I_s \cup \left\{i_j \mid j\in  \mathcal J'\right\}$. Using the fact that $|\mathcal J'| =  \lfloor (t - \mathcal |I_s|) /2 \rfloor$ when $s< m - 1$ we obtain
\begin{align}
   |\mathcal I_{s+1}| =  |\mathcal I_s| + \lfloor (t - |\mathcal I_s|) /2 \rfloor 
   = t - (t -  |\mathcal I_s| - \lfloor (t - |\mathcal I_s|) /2 \rfloor ) ) 
   = t - \lceil (t - |\mathcal I_s|) /2 \rceil \stackrel{(X1)}{=}  t - \lceil t /2^{s+1} \rceil. \label{eq:induc_step_p1}
  \end{align}
If $s = m - 1$ if follows from the invariant that $|\mathcal I_s| = t - 1$ and therefore that
  $|\mathcal I_{s+1}| = t$.
Note that the invariant $(X2)$ holds directly by the construction of the new $\sigma$-walk $P_{i_j}$.

In order to show $(X3)$ let $\mathcal T'_A \subseteq \mathcal T^s_A$ and $\mathcal T'_B \subseteq \mathcal T^s_B$ be the subsets which contain all trees that are rooted at  vertices from
$\cup_{i \in \mathcal I_{s+1}} \{a_i, b_i\}$. For $L_A = \cup_{T \in T'_A} \mathcal L(T)$ and $L_B = \cup_{T \in T'_B} \mathcal L(T)$ it follows from \eqref{eq:induc_step_p1} that $|L_A| = |L_B| = 2^{s}  \cdot \lceil t /2^{s+1} \rceil$.
 Using Claim \ref{claim:2} bellow we conclude that there exist an $L_A$-saturating
 $(2,\sigma(s + 1))$-matching $M_A$ from $L_A$ to $A_{s+1}$, and a $L_B$-saturating
$(2,\bar\sigma(\ell-s))$-matching $M_B$ from $L_B$ to $B_{s+1}$.
For every $\sigma^j$-tree $T \in \mathcal T^s_A$ we denote by $T^+$ the $\sigma^j$-tree obtained by extending $T$ with the arcs of the matching $M_A$ incident to  $\mathcal L(T)$.
Similarly, for every $\bar \sigma^j$-tree $T \in \mathcal T^s_B$ we denote by $T^+$ the $\bar \sigma_{j+1}$ tree obtained by extending $T$ with the arcs of the matching $M_B$ incident to  $\mathcal L(T)$.
Finally, let $\mathcal T^{s+1}_A = \cup_{T \in \mathcal T'^s_A} T^+$ and let
$\mathcal T^{s+1}_B = \cup_{T \in \mathcal T'^s_B} T^+$. It follows from our construction that $T^{s+1}_A$ and $T^{s+1}_B$ satisfy $(X3)$ and $(X4)$.
This completes the proof of the lemma.

\begin{claim}\label{claim:2}
For every $0\leq i\leq m-1$, every $\tau \in \{+, -\}$ and  $X \in \{A, B\}$ the following holds. For every subset $S\subseteq X_i$ of size  $|S| \leq |X_{i+1}| / 8$  there is a $(2, \tau)$-matching
 from $S$ to $X_{i+1}$  that saturates $S$.
\end{claim}
\begin{proof}[Proof of Claim \ref{claim:2}]

We prove the existence of a $(2, +)$-matching from $S$ to $X_{i+1}$ that saturates $S$ and the proof for such a
$(2, -)$-matching follows similarly.
Using Hall's Theorem (see e.g. \cite{West}), it is sufficient to
prove that for every $S' \subseteq S$ it holds $|N^{+}(S',
X_{i+1})| \geq 2|S'|$. Assume the existence of a subset $S'
\subseteq S$ that violates Hall's condition, i.e. $|N^{+}(S',
X_{i+1})| < 2|S'|$ .
If $|S'| \leq \frac{\log^2 n}{3p}$
then, since $|S' \cup N^{+}(S', X_{i+1}) | \leq \frac{\log^2 n}{p}$, by property \PXEdges{}
we obtain that
\begin{equation}
e_D(S' \cup N^{+}(S', X_{i+1}))\leq 3|S'| \log^{2.1} n. \label{eq:e_D_bound}
\end{equation}
Moreover, it follows from \eqref{eq:good-deg-in-parts} that
$d^{\pm}(s',X_{i+1}) \geq (1/2 + \alpha /2 ) p |X_{i+1}|$, for all $s' \in S'$. All in all, we get
\begin{align}
e_D(S', N^{+}(S', X_{i+1})) & \geq (1/2 + \alpha/2) p|S'| |X_{i+1}|
 > 3|S'| \log^{2.2}n,
\label{eq:lo_bound_matching}
\end{align}
where the second inequality follow from $|X_{i+1}| \geq \frac{6\log^{2.2}n}{p}$. The last inequality together with \eqref{eq:e_D_bound} leads to a contradiction.

If on the other hand
$|S'| > \frac{\log^2 n}{3p}$ and
 $|N^{+}(S', X_{i+1})| < 2|S'|$,
it follows from \PXYEdges{} and the assumption on $S'$ that
$$
e_D(S', N^{+}(S', X_{i+1})) \leq (1+ \alpha/2) 2 p |S'|^2 . $$
However by \eqref{eq:lo_bound_matching} and the assumption   $|X_{i+1}| \geq 4 |S|$ we have  $e_D(S', N^{+}(S', X_{i+1})) \geq (2 + 2 \alpha) p |S'|^2 $. Therefore, by combining the previous inequalities we obtain
$$
(2+ \alpha)  p |S'|^2  \geq e_D(S', N^{+}(S', X_{i+1})) 
\geq (2 + 2\alpha) p |S'|^2,
$$
which is a contradiction.
\end{proof}

\begin{claim}\label{claim:3}
For every $1 \leq s \leq m -1 $ and every subset
$S\subseteq K$ such that  $|S| \geq \frac{\log^2 n}p$ and $X \in \{A, B\}$ the
following holds:
\begin{align} \label{eq:expansion}
|N^+(S,R'_X)|,|N^-(S,R_X')|\geq (1/2+ \alpha/40)|R'_X|,
\end{align}
where $R'_X := R_X \setminus \cup_{P \in \mathcal P_s} V(P)$.
\end{claim}
\begin{proof}[Proof of Claim \ref{claim:3}]

We prove the claim for $|N^+(S, R'_X)|$ as the proof for $|N^-(S, R'_X)|$ follows analogously.
By Equation \eqref{eq:good-deg-in-parts} we know that for $X \in \{A, B\}$, for any $v \in \{a_i, b_i\}_{i=1}^t \cup K$ the following holds at each step $s$:
$$
d^{\pm}(v, R_X) \geq (1/2 +  \alpha/2) p |R_X|.
$$
Applying Lemma \ref{lemma:big-expansion} to $S$ (as $X$),
$R_X$ (as $Y$) we get
$
|N^+(S, R_X)| \geq (1/2 + \alpha / 20) |R_X|.
$
Since $|\cup_{P \in \mathcal P_s} V(P)| \leq t \ell$  and $|R_X|=  \omega(t \ell)$ we conclude that
\begin{align*}
|N^+(S,R'_X)| & \geq (1/2+ \alpha / 20)|R_X| - t \ell 
 \geq (1/2+ \alpha /40)|R'_X|.
\end{align*}
\end{proof}
 \end{proof}

\section{Proof of the Absorbing Lemma}
\label{sec:absorbing_lemma}

In this section we prove Lemma \ref{lemma:absorbing}.
 A main ingredient in our proof is the concept of an \emph{absorber}. Roughly speaking,
in our setting (that is, finding a directed Hamilton cycle in a
digraph $D$) an absorber $A_x$ for a vertex $x$ is a digraph which contains $x$ and two designated vertices
$s_x$ and $t_x$, such that $A_x$ contains two $s_xt_x$-paths: one
which consists of all vertices in $V(A_x)$ and the other which consists of all
vertices in $V(A_x)\setminus\{x\}$. 
\begin{definition}
\label{def:absorber} Let $\ell_x$ be an integer and $A_x$ a digraph of size $\ell_x + 1$. Then for some distinct verties $x, s_x, t_x \in V(A_x)$, the digraph $A_x$ is called an
\emph{absorber} for a vertex $x$ with \emph{starting} point $s_x$
and a \emph{terminal} point $t_x$, if it contains an $s_xt_x$-path
$P_x$, referred to as the \emph{non-absorbing} path, of length
$\ell_x - 1$ that does not contain $x$, and an $s_xt_x$-path $P_x'$ of
length $\ell_x$ which is referred to as the \emph{absorbing} path.
\end{definition}

\begin{figure}[h]
\centering
\begin{tikzpicture}[%
  >=stealth,
  shorten >=1pt,
  shorten <=1pt
]
\tikzstyle{vertex} = [circle, fill, inner sep=0pt, minimum size=5pt]
\tikzstyle{edge} = [->]
\tikzstyle{edged} = [->,dashed]
\tikzstyle{top} = [bend left=40]
\tikzstyle{bot} = [bend right=40]
\node (v1) [vertex, label=$x_s$] at (0,0) {};
\node (v2) [vertex, label=$x$] at (1,0) {};
\node (v3) [vertex, label=below:$s^x_1$] at (2,0) {};
\node (v4) [vertex, label=below:$t^x_1$] at (3,0) {};
\node (v5) [vertex, label=below:$s^x_2$] at (4,0) {};
\node (v6) [vertex, label=below:$t^x_2$] at (5,0) {};
\node (v7) [vertex] at (6,0) {};
\node (v8) [vertex] at (7,0) {};
\node (v9) [vertex, label=below:$s^x_i$] at (8,0) {};
\node (v10) [vertex, label=below:$t^x_i$] at (9,0) {};
\node (v11) [vertex] at (10,0) {};
\node (v12) [vertex] at (11,0) {};
\node (v13) [vertex, label=below:$s^x_{2k}$] at (12,0) {};
\node (v14) [vertex, label=below:$t^x_{2k}$] at (13,0) {};
\node (v15) [vertex, label=$x_t$] at (14,0) {};

\draw[edge] (v1) edge [bend right=50] (v5);
\draw[edge] (v6) edge [bend right=50] (v9);
\draw[edge] (v10) edge [bend right=50] (v13);

\draw[edge] (v4) edge [top] (v7);
\draw[edge] (v8) edge [top] (v11);
\draw[edge] (v12) edge [top] (v15);

\draw[edge] (v14) edge [bend right=30] (v3);

\draw[edge] (v1) -- (v2);
\draw[edge] (v2) -- (v3);
\draw[edged] (v3) -- (v4);
\draw[edge] (v4) -- (v5);
\draw[edged] (v5) -- (v6);
\draw[edge] (v6) -- (v7);
\draw[edged] (v7) -- (v8);
\draw[edge] (v8) -- (v9);
\draw[edged] (v9) -- (v10);
\draw[edge] (v10) -- (v11);
\draw[edged] (v11) -- (v12);
\draw[edge] (v12) -- (v13);
\draw[edged] (v13) -- (v14);
\draw[edge] (v14) -- (v15);

\draw ($(v8)+(-0.5,0.95)$) rectangle ($(v11)+(0.5,-0.9)$);
\end{tikzpicture}
\caption{The absorber for $k=3$. The cycle $C$ is drawn with solid arrows. The dashed arrows represent directed paths of arbitrary length. The part inside the rectangle can be repeated to obtain absorbers for larger $k$.}
\label{fig:absorber}
\end{figure}
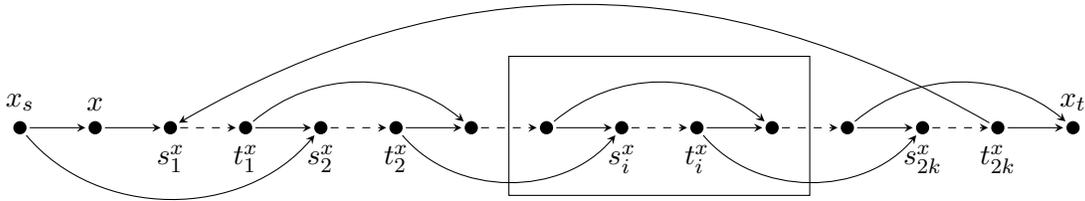

In the following lemma we describe the structure of our absorber.
\begin{lemma}
\label{lemma:absorbers_ax} Let $k$ and $\ell$ be integers and consider a digraph $A_x$ of size
$3+2k(\ell+1)$ constructed as follows:
\begin{enumerate}[$(i)$]
\item $A_x$ consists of a cycle $C$ of length $4k+3$ with an orientation of the edges and labeling of the vertices as shown in Figure \ref{fig:absorber}, and
\item $A_x$ 
contains $2k$ pairwise disjoint directed $s^x_it^x_i$-paths $P_i$ (for each $i \in [2k]$), each of which is of length $\ell$.
\end{enumerate}
Then $A_x$ is an absorber for the vertex $x$.
\end{lemma}

\begin{proof}
It is easy to see that
$$ P_x' := x_s, x, s_1^x, P_1, t_1^x, s_2^x, P_2 \ldots, t_{2k}^x, x_t$$
is an absorbing path. On the other hand, the path
$$
P_x := x_s, s_2^x, P_2, t_2^x, s_4^x, P_4, \ldots, s_{2k}^x, t_{2k}^x,
s_1^x, P_1, t_1^x, s_3^x, P_3, \ldots, t_{2k-1}^x, x_t
$$
uses all vertices except $x$, thus it is a non-absorbing path. We refer the reader to Figure \ref{fig:absorber} for clarification.
\end{proof}

The proof of the Absorbing Lemma consists of two main steps.
First, we show how to constuct an absorber $A_x$ for each $x \in V_1$ such that the non-absorbing path of $A_x$ is contained in $V_2 \cup V_3$ and $V(A_x) \cap V(A_x') = \emptyset$ for $x \neq x'$.
Second, using Lemma \ref{lemma:connecting-lemma} we connect non-absorbing paths of each absorber into one long path using vertices from $V_4$.

We build the absorbers $A_x$ in $D$ by first finding the cycle of
the absorber and then connecting all the designated pairs of
vertices via directed paths. To do so we use Lemma \ref{lemma:connecting-lemma} (note that a cycle is a
$\sigma$-walk, for some $\sigma$). 

\begin{proof}[Proof of Lemma \ref{lemma:absorbing}]

Let $k:=3\lceil \log n\rceil$ and let $A_x$ be an absorber given by Lemma \ref{lemma:absorbers_ax}. Recall that $A_x$ contains a cycle $C_x$ of length $4k+3=12\lceil \log n\rceil+3$ with a prescribed orientation
$\sigma$ and
$2k=6\lceil \log n\rceil$ disjoint directed paths $P_1, \dots, P_n$, of length
$\ell:=10\lceil \log n\rceil$, connecting the pairs of designated
vertices $(s_1^x, t_1^x), \dots, (s_{2k}^x, t_{2k}^x)$ on the cycle.
In order to find for each $x\in V_1$ such a cycle, we apply Lemma~\ref{lemma:connecting-lemma} to the 
set $V_2$ (as $K$), $\ell=12\lceil \log n\rceil+3$, $t = |V_1|$ and a family of pairs 
$\{(x,x)\}_{x \in V_1}$
and thereby obtain for every $x\in V_1$ 
a $\sigma$-walk of length $4k+3$ from $x$ to itself, which is  a
cycle $C_x$ as required for the absorber. Moreover, all obtained cycles
$\left\{C_x\right\}_{x\in V_1}$ are disjoint and contain
(apart from the absorbing vertices) only vertices in $V_2$. 
Note that we can apply Lemma \ref{lemma:connecting-lemma} as $V_2 = \omega(  |V_1| \log n)$, $t = |V_1| \geq \frac{4 \log^2 n}{p}$ and 
by property \PMinDeg{} we have that $(ii)$ and $(iii)$ from Lemma \ref{lemma:connecting-lemma} are true.

Next, using the vertices in $V_3$, for each $x\in V_1$ we connect
the pair of designated vertices $(s_i^x, t_i^x)$ on a cycle $C_x$
by a directed path. For this aim we apply
Lemma~\ref{lemma:connecting-lemma} to $V_3$ (as $K$) with $\ell:=10\lceil
\log n \rceil$, $t = 2k |V_1|$,
and $ \{ (s_i^x, t_i^x) \mid x \in V_1, 1 \le i \le 2k \}$
to find all the required paths to complete the absorbers. We can indeed do that as 
$|V_3| = \omega( |V_1| \log^2 n)$, $t = 2k |V_1| \geq \frac{4 \log^2 n}{p}$ and 
by property \PMinDeg{} we have that $(ii)$ and $(iii)$ from Lemma \ref{lemma:connecting-lemma} are true.
 
Finally, we build a directed path which contains all the non-absorbing
paths of the absorbers. To do so, recall that by
Definition~\ref{def:absorber} every absorber $A_x$ has a start
vertex $s_x$ and a terminal vertex $t_x$. Let us arbitrarily  enumerate vertices from $V_1$ as $V_1 = \{x_1, \ldots, x_h\}$, where $h = |V_1|$.
 Apply Lemma~\ref{lemma:connecting-lemma} to $V_4$ (as $K$), $\ell=10\lceil \log n\rceil$, $t = |V_1| - 1$ and 
a family of pairs $ \{(t_{x_i}, s_{x_{i+1}})\}_{i \in [h-1]}$
to find the required paths of length $\ell$ that connect all non-absorbing paths of the absorbers into one directed path $P^*$.
Again, we are allowed to apply the lemma as $|V_4| = \omega( |V_1| \log n)$ and 
by property \PMinDeg{} we have that $(ii)$ and $(iii)$ from Lemma \ref{lemma:connecting-lemma} are true.

It is now easy to see that the path $P^*$ has the required properties. Let $W \subseteq V_1$ be an arbitrary subset of $V_1$ and let $\{A_w \mid w \in W\}$ be the set of absorbers for vertices in $W$. By the definition of absorber for each $A_w$ there is an absorbing path starting and ending at the same vertices as the non-absorbing path, but which contains vertex $w$ as well. By replacing non-absorbing paths of $\{A_w \mid w \in W\}$ in $P^*$ with corresponding absorbing paths we obtain a path $P^*_W$ which has the same endpoints as $P^*$ and $V(P^*_W) = V(P^*) \cup W$. 
\end{proof}


\bibliographystyle{abbrv}
\bibliography{ResilienceDnp}

\end{document}